\documentclass[12pt, a4paper]{article}
\usepackage[english]{babel}
\usepackage{amsmath,amssymb,amsfonts,amscd}
\usepackage{amsthm}
\usepackage[T1]{fontenc}
\usepackage{latexsym}
\usepackage{amssymb,amsbsy,amsmath,amscd}
\usepackage{amsmath,amssymb,amsbsy, amsmath, amsfonts, amscd}
\usepackage[mathcal]{eucal}
\usepackage[pdftex]{color,graphicx}
\usepackage{latexsym}
\usepackage{fancyhdr}
\usepackage{cite}
\newtheorem{defi}{Definition}[section]
\newtheorem{lemma}[defi]{Lemma}
\newtheorem{proposition}[defi]{Proposition}
\newtheorem{theorem}[defi]{Theorem}
\newtheorem{corollary}[defi]{Corollary}

\newcommand{\cc}{\mathcal}
\newcommand{\R}{\mathbb R}
\newcommand{\N}{\mathbb N}
\newcommand{\C}{\mathbb C}
\newcommand{\w}{^W\!\!D}

\begin{document}

\title{Semilinear equations associated with Dunkl Laplacian
}


\author{Mohamed Ben Chrouda$^a$
\and
Khalifa El Mabrouk$^b$
\and
Kods Hassine$^c$
}

\maketitle
\begin{center}{\scriptsize{\emph{$^a$Dep. of Mathematics, High institut of computer sciences and mathematics, 5000, Monastir, Tunisia }}}\end{center}
\begin{center}{\scriptsize{\emph{$^b$Dep. of Mathematics, High School of Sciences and Technology, 4011 Hammam Sousse, Tunisia }}}\end{center}
\begin{center}{\scriptsize{\emph{$^c$Dep. of Mathematics,  Faculty of Science of Monastir, 5000 Monastir, Tunisia }}}\end{center}

\begin{abstract}
Let $\Delta_k$ be the Dunkl Laplacian on $\R^d$ associated with a reflection group $W$ and a multiplicity function $k$. The purpose of this paper is to establish necessary and sufficient condition
under which there exists a positive solution of the  equation
$\Delta_ku=\varphi(u)$
in the unit ball of $\R^d$ as well as in the whole space $\R^d$.
\end{abstract}
\section{Introduction}
The Dunkl Laplacian is the sum of a second order differential operator and a difference term associated with a multiplicity function $k$ and a reflection group $W$. An important motivation to study the Dunkl Laplacian rises from its relevance for the analysis of certain exactly solvable models of  mechanics, namely the Calogero-Moser-Sutherland type (see \cite{van,khi,mlv}). Since it's introduction by C. F. Dunkl in~\cite{dunkl}, analysis of Dunkl theory has been the subject of many articles  and  it has  deep and fruitful interactions with various mathematic's fields namely  Fourier analysis and special function \cite{dej,trim,xu}, algebra (double affine Hecke algebras~\cite{khon}) and probability theory (Feller processes with jumps~\cite{gy,demni}).
\par By introducing some potential theoretical concepts relative to $\Delta_k$, we investigate in this paper the equation
\begin{equation}\label{introeq1}
\Delta_ku=\varphi(u)
\end{equation}
 in open balls of $\R^d$ with center $0$, where the perturbation $\varphi:[0,\infty[\rightarrow[0,\infty[$ is assumed to be locally Lipschitz and  non decreasing  such that $\varphi(0)=0$. If the multiplicity function $k$ is identically vanishing, the operator $\Delta_k$ reduces to the classical Laplace operator~$\Delta$. In this case, the elliptic equation
$
\Delta u=\varphi(u)
$
have been studied by several authors in different classes of domains (bounded and unbounded) and under various conditions on the perturbation $\varphi$ (see \cite{bresiz,dyn00,khalifabounded,lairwood,lazer,kel,oss}).
\par For arbitrary multiplicity function $k$, it was shown in \cite{rosler1} that $\Delta_k$ generates a positive strongly continuous contraction semigroup (which reduces to the classical Brownian semigroup if $k=0$). This fact gives rise to a Hunt process, called Dunkl process, and so to a corresponding family of harmonic kernels~$(H_V)_V$. The Dunkl process has a discontinuous paths (see \cite{gy}). So, in virtue of the general theory of balayage spaces \cite{hansen}, Dunkl process generates a Balyage space and not a harmonic space. This yields that for every bounded open set $V$ and every $x\in V$ the harmonic measure $H_V(x,\cdot)$ is not necessarily supported by the Euclidean boundary~$\partial V$ of~$V$ but it may live on the entire complement~$V^c:=\R^d\backslash V$.
\par Throughout this paper we assume that $k$ is strictly positive. We shall show that, for each bounded open subset $V$ of $\R^d$, the harmonic measure $H_V(x,\cdot)$ relative to $V$ is supported by a compact set of $V^c$. Moreover, $H_B(x,\partial B)=1$ for every open ball $B$ centered at the origin of $\R^d$ and every $x\in B$. This fact allows us to investigate the Dirichlet problem
\begin{equation}\label{pbd}\left\{
\begin{array}{lccl}
\Delta_ku&=&\varphi(u) &\mbox{ in }{B}\\
u&=&f&\mbox{ on }\partial B
\end{array}\right.
\end{equation}
where $f$ is in $C^+(\partial B)$ the set of all nonnegative continuous functions on $\partial B$. More precisely, we prove that the function $H_Bf$ defined on $\overline{B}$ by
 $$H_Bf(x)=
 \int_{\partial B}f(y)H_B(x,dy)\;\mbox{ if }\;x\in B\quad \textrm{ and } \quad
 H_Bf(x)=f(x)\;\mbox{ if }\;x\in\partial B,
 $$
 is the unique continuous  extension $u$ of $f$ on $\overline{B}$ satisfying $\Delta_k u=0$ in $B$. This means that $H_Bf$ is the unique solution of~(\ref{pbd}), replacing $\varphi$ by 0.
\par Assuming that $\varphi$ is not trivial, we prove that for each $f\in C^+(\partial B)$, Problem~(\ref{pbd}) admits a unique solution in $C(\overline B)$. In fact, we first show that $u$ satisfies~(\ref{pbd}) if and only if
$$u+G^k_B(\varphi(u))=H_Bf,$$
where $G^k_B$ is the Green operator on $B$. Then, by a compactness argument of $G^k_B$ we show that the map $u\mapsto H_Bf-G^k_B(\varphi(u))$ admits  a fixed point $u\in C(\overline B)$ and consequently, $u$ is a solution of~(\ref{pbd}). The uniqueness of such solution is established by mean of a minimum principle. By giving some properties of radially symmetric solutions, we prove that Eq~(\ref{introeq1}) admits a positive blow up solution (that is a positive solution on $B$ such that $u=\infty$ on~$\partial B$) if and only if there exists a constant $a>0$ such that
 $$\int_a^\infty\frac{dt}{\sqrt{\int_0^t\varphi(s)ds}}<\infty.$$
 \par The question whether or not there exists a positive solution of Eq~(\ref{introeq1}) on the whole space $\R^d$ is also treated in this paper. In fact, we prove that 
 \begin{equation}\label{keloss}
 \int_a^\infty\frac{dt}{\sqrt{\int_0^t\varphi(s)ds}}=\infty\quad \mbox{ for some } a>0
 \end{equation}
 is a necessary and sufficient condition on $\varphi$ to obtain a nonnegative solution $u$ of Eq~(\ref{introeq1}) in $\R^d$. Moreover  $\lim_{|x|\rightarrow\infty}u(x)=\infty$. Notice that (\ref{keloss}) is called Keller-Osserman condition in the literature due to J.B.Keller \cite{kel} and R.Osserman~\cite{oss} who proved that (\ref{keloss}) characterize the existence of a positive solution of the equation $\Delta u=\varphi(u)$ in $\R^d$.

 \par As mentioned above, our approach is based in potential theoretical tools. Namely, we establish several properties for harmonic measures and  the Green kernel which are the principal materials used in the study of Eq~(\ref{introeq1}) and problem~(\ref{pbd}).

\section{Notations and preliminaries}
For every subset  $F$ of $\R^d$, let $\cc{B}(F)$ be the set of all Borel measurable functions on $F$ and
 let $1_F$ be  the indicator function of $F$. Let ${C}(F)$ be the set of all continuous real-valued functions on $F$, $C^k(F)$ is the class of all functions that are $k$ times   continuously differentiable on $F$ and $C_0(F)$ is the set of all continuous  functions on $F$ such that $u=0$ on $\partial F,$ which means that $\lim_{x\to z}u(x)=0$ for all $z\in\partial F$ and $\lim_{x\to\infty}u(x)=0$ if $F$ is unbounded. We denote by  $\mathcal{C}_c^\infty(F)$  the set of all infinitely differentiable functions on $F$ with compact support. If $\cc{G}$ is a set of numerical
  functions then $\cc{G}^+$ (respectively $\cc{G}_b$) will denote the class of all functions in $\cc{G}$ which
  are nonnegative (respectively bounded). The uniform  norm will be denoted by $\left\|\cdot\right\|.$

For every $\alpha\in \R^d\setminus\{0\}$, let $H_\alpha$ be the hyperplane orthogonal to $\alpha$ and let $\sigma_\alpha$ be the reflection in $H_\alpha$, i.e.,
$$\sigma_\alpha(x):=x-2\frac{\langle \alpha,x\rangle}{|\alpha|^2}\alpha,$$
where $\langle\cdot,\cdot\rangle$ denotes the usual inner product on $\R^d$ and $|\cdot|$ is the associated norm. A finite subset $R$ of $\R^d\setminus\{0\}$ is called a root system if $R \cap \R\cdot\alpha = \{\pm\alpha\}$ and $\sigma_\alpha(R)=R$ for all $\alpha\in R$. For a given root system $R$, the reflection $\sigma_\alpha,\;\alpha\in R$, generates a finite group $W$ called reflection group associated with $R$.
A function $k : R\rightarrow \R_+$ is called a
multiplicity function if it satisfies $k(\sigma_\alpha\beta) = k(\beta)$, for every  $\alpha,\beta\in R$.
Throughout this paper we fix a root system $R$ and  a multiplicity function $k$.
We consider the differential-difference operators $T_i,\; 1\leq i\leq d,$  defined in \cite{dunkl1}  for every  $u\in C^1(\R^d)$ by
$$
T_iu(x)=\frac{\partial u}{\partial x_i}(x)+\frac 12\sum_{\alpha\in R}k(\alpha)\alpha_i\frac{u(x)-u(\sigma_\alpha x)}{\langle\alpha,x\rangle},
$$
and called  Dunkl operators in the literature. The Dunkl Laplacian $\Delta_k$,  is  the sum of squares of Dunkl operators
$$
\Delta_k:=\sum_{i=1}^dT_i^2.
$$
It is given explicitly, for  $u\in C^2(\R^d)$, by
\begin{equation}\label{lapd}
\Delta_ku(x)=\Delta u(x)+\sum_{\alpha\in R}k(\alpha)\left(\frac{\langle\nabla u (x),\alpha\rangle}{\langle\alpha,x\rangle}-\frac{|\alpha|^2}{2}\frac{u(x)-u(\sigma_\alpha(x))}{\langle\alpha,x\rangle^2}\right).
\end{equation}
Likewise  the classical Laplace operator $\Delta$, the Dunkl Laplacian has the following  symmetry property: For   $u\in C^2(\R^d)$ and $v\in C_c^2(\R^d)$
\begin{equation}\label{ti}
\int_{\R^d} \Delta_k u(x) v(x) w_k(x)\,dx=\int_{\R^d} u(x)\Delta_kv(x)w_k(x)\,dx,
\end{equation}
where $w_k$ is the homogeneous weight function defined on~$\R^d$ by
$$
w_k(x)=\prod_{\alpha\in
R}|\langle x,\alpha\rangle|^{k(\alpha)}.
$$
A  fundamental result in Dunkl theory is the existence of an intertwining operator $V_k:C^\infty(\R^d)\rightarrow C^\infty(\R^d)$ between  the  usual laplacian $\Delta$ and  Dunkl Laplacian i.e., $\Delta_kV_k=V_k \Delta .$ We refer to \cite{dunkl2,rosler2,trim} for more details about  the intertwining operator. By mean of $V_k$, there exists a counterpart
of the usual exponential function, called  Dunkl kernel $E_k(\cdot,\cdot)$ which is defined for every  $y\in \C^d$ and $x\in\R^d$ by
$$E_k(x,y)=V_k\left(e^{\langle \cdot,y\rangle}\right)(x).$$
It is clear from (\ref{lapd}) that   if  $k$ vanishes identically then  the Dunkl Laplacian reduces to the classical Laplacian $\Delta$. In this case  the intertwining operators $V_k$ is the identity operator and so $E_k$ reduces to the classical exponential function. Notice that $E_k$ is symmetric and positive on $\R^d\times\R^d$ and satisfies   $E_k(\lambda y,x)=E_k(y,\lambda x)=$ for every $\lambda \in \C$.
  \par In all this paper we assume that
   $$m:=d+\sum_{\alpha\in R}k(\alpha)>2.$$
Let $p_t^k$ be the Dunkl heat kernel, introduced in \cite{rosler1}, defined for every $t>0$ and every $x,y\in\R^d$ by
\begin{equation}\label{intrep}
p_t^k(x,y)=\frac{c_k^2}{2^m}\int_{\R^d}e^{-t|\xi|^2}E_k(-ix,\xi)E_k(iy,\xi)w_k(\xi)d\xi,
\end{equation}
where $$
c_k=\left(\,\int_{\R^d}e^{-|y|^2}w_k(y)\,dy\right)^{-1}.
$$
For every $x,y\in \R^d$, $0<p_t^k(x,y)$, $p_t^k(x,y)=p_t^k(y,x)$ and
\begin{equation}\label{ptk}
p_t^k(x,y)\leq \frac{c_k}{(4t)^{\frac{m}{2}}}
\exp\left(-\frac{\left(|x|-|y|\right)^2}{4t}\right).
\end{equation}
Also, for every $x\in\R^d$, the function $(t,y)\to p_t^k(x,y)$ solves  the generalised heat equation $\partial_tu-\Delta_ku=0$ on $]0,\infty[\times \R^d$. More precisely, the following holds

\begin{equation}\label{chal}
\frac{\partial}{\partial t}p_t^k(x,y)= \Delta_k\left(p_t^k(\cdot,y)\right)(x)=\Delta_k\left( p_t^k(x,\cdot )\right)(y).
\end{equation}
For every  $f\in C_0(\R^d)$ and $t>0$ let
$$
P^k_tf(x)=\int_{\R^d}p_t^k(x,y)f(y)w_k(y)\,dy,\;x\in\R^d.
$$
Then $(P_t^k)_{t>0}$ forms   a positive strongly continuous contraction semigroup on $C_0(\R^d)$ of generator $\Delta_k$.
This fact  yields the existence of a  Hunt process $(X_t,P^x)$ (see \cite[Theorem I.9.4]{blum}), called Dunkl process, with state space $\R^d$ and  transition kernel
$$
P_t^k(x,dy)= p_t^k(x,y)w_k(y)\,dy.
$$
\section{Harmonic Kernels}
For every  bounded open subset $D$ of $\R^d$, we denote by  $\tau_D$   the first exit time from $D$ by $(X_t),$
 i.e.,
$$
\tau_D=\inf\left\{t>0;X_t\notin D\right\}.
$$
\begin{lemma}\label{taus}
Let $D$ be a bounded open set. Then, for every $x\in D$, $P^x\left(0<\tau_D<\infty\right)=1$.
\end{lemma}
\begin{proof} Let $x\in D$.
 Since Dunkl process has right continuous paths, we immediately conclude that $P^x(0<\tau_D)=1$.
 Let $r>0$ such that $D\subset B_r$ the ball of center $0$ and radius $r$. Clearly,
\begin{eqnarray*}
 E^x[\tau_D]\leq E^x[\tau_{B_r}]&=& E^x\left[\int_0^{\tau_{B_r}}1_{B_r}(X_t)\,dt\right]\\
   &\leq&  \int_0^{\infty}E^x[1_{B_r}(X_t)]dt\\
   &=&\int_0^\infty \int_{B_r} p_t^k(x,y)w_k(y)\,dy\, dt.
   \end{eqnarray*}
   So, to prove that $P^x(\tau_D<\infty)=1$, it will be sufficient to shows that
   $$
   \int_0^\infty \int_{B_r} p_t^k(x,y)w_k(y)\,dy\, dt<\infty.
   $$
Using spherical coordinates and seeing that the function $w_k$ is homogenous of degree $m-d$, it follows from the integral representation (\ref{intrep}) of $p^k_t$ that, for every $y\in\R^d$,
\begin{eqnarray*}
p_t^k(x,y)&=&\frac{c_k^2}{2^m}\int_{0}^\infty\int_{S^{d-1}} e^{-ts^2} E_k(-ix,s\xi) E_k(iy,s\xi)w_k(\xi)s^{m-1}\sigma(d\xi) ds,
\end{eqnarray*}
where  $\sigma$ denotes the surface area measure on the unit sphere $S^{d-1}$ of $\R^d$. Therefore,
$$\int_0^\infty p_t^k(x,y)dt= \frac{c_k^2}{2^m}\int_{0}^\infty\int_{S^{d-1}}  E_k(-ix,s\xi) E_k(iy,s\xi)w_k(\xi)s^{m+1}\sigma(d\xi) ds.$$
Using again spherical coordinates and then applying Fubini's theorem, we get
\begin{eqnarray*}
\int_0^\infty \int_{B_r} p_t^k(x,y)w_k(y)\,dy\, dt
&=& \int_0^r\int_{S^{d-1}} \left(\int_0^\infty p_t^k(x,uy)dt\right)w_k(y)u^{m-1}\sigma(dy)du\\
&=&\frac{c_k^2}{2^m}\int_0^r\int_0^\infty \int_{S^{d-1}}\left(\int_{S^{d-1}}E_k(iuy,s\xi)w_k(y)\sigma(dy)\right)\\
& &E_k(-ix,s\xi) w_k(\xi)s^{m+1}u^{m-1}\sigma(d\xi)dsdu.
\end{eqnarray*}
On the other hand, we recall from \cite{rosler3} that
$$
 \int_{S^{d-1}}E_k(iz,y)w_k(y)\sigma(dy)=2^{\frac m2}c_k^{-1}\frac{J_{\frac m2 -1}(|z|)}{ |z|^{\frac m2-1}},
 $$
 where $J_{\frac m2 -1}$ is the Bessel function of index $\frac m2-1$ given by
 $$
J_{\frac m2 -1}(z):=\left(\frac{z}{2}\right)^{\frac m2-1}\sum_{n=0}^\infty\frac{(-1)^nz^{2n}}{4^nn!\Gamma(n+\frac m2)}.
$$
  Whence
\begin{eqnarray}
\int_0^\infty \int_{B_r} p_t^k(x,y)w_k(y)\,dy\, dt&=&\int_0^r \frac{u^{m-1}}{(u|x|)^{\frac m2-1}}\left(\int_0^\infty J_{\frac{m}2-1}(s|x|) J_{\frac m2-1}(us)s^{-1}ds\right) du\nonumber\\
&=& \frac{1}{m-2} \int_0^r u^{m-1}\left(\max(u,|x|)\right)^{2-m}\,du\label{locb1}\\
&=&
\frac{1}{m-2}\left(\frac{|x|^2}{m}+\frac{r^2-|x|^2}{2}\right).\nonumber
\end{eqnarray}
 In order to get~(\ref{locb1}) one should think about  formula in \cite[p.100]{magnus}.
\end{proof}

For every bounded open set $D$,
we define
  $$
\w:=\cup_{w\in W}w(D)\quad \mbox{and}\quad \Gamma_D:=\overline{\w}\setminus D.
$$
That is, $\w$ is the smallest open bounded set  containing $D$ which is invariant under the reflection group $W$. In the following theorem, we show that if the process start from $x\in D$ then, at the first exit time from $D$, it can not be anywhere outside the domain $D$ but it should be in the compact $\Gamma_D$.

\begin{theorem}\label{support}
Let $D$ be a bounded open subset of $\R^d$. Then for every $x\in D$,
\begin{equation}\label{gmd}
P^x\left(X_{\tau_D}\in\Gamma_D\right)=1.
\end{equation}
In particular, if $D$ is $W$-invariant, i.e., $\w=D$, then $\Gamma_D=\partial D$ and therefore $P^x\left(X_{\tau_D}\in\partial D\right)=1$.
\end{theorem}
\begin{proof} Let $x\in D$  and consider the function $\digamma$ defined for every $y,z\in\R^d$ by $\digamma(y,z)=0$ if $z\in\{\sigma_\alpha y;\alpha\in R\}$ and $\digamma(y,z)=1$ otherwise. Let
 $$
 Y_t:=\sum_{s<t}1_{\{X_{s^-}\neq X_s\}}\digamma(X_{s^-},X_s), \quad t>0.
 $$
 It follows from \cite[Proposition 3.2]{gy} that for every $t>0$, $P^x(Y_t=0)=1$ and consequently
  $$
  P^x\left(1_{\{X_{s^-}\neq X_s\}}\digamma(X_{s^-},X_s)=0;\forall s>0\right)=1.
  $$
  Then, since  $P^x(0<\tau_D<\infty)=1$ we deduce that
  $$
 P^x\left(1_{\{X_{\tau_D^-}\neq X_{\tau_D}\}} \digamma(X_{\tau_D^-},X_{\tau_D})=0\right)=1.
  $$
  On the other hand, seeing that $X_{\tau_D^-}\in \overline D$ on $\{0<\tau_D<\infty\}$ we have
  $$
  \left\{X_{\tau_D}\not\in \Gamma_D, 0<\tau_D<\infty\right\}\subset
  \left\{ 1_{\{X_{\tau_D^-}\neq X_{\tau_D}\}} \digamma(X_{\tau_D^-},X_{\tau_D})=1  \right\}.
  $$
  This finishes the proof.
 \end{proof}

For every bounded open set $D$ and every $x\in \R^d$, let $H_D(x,\cdot)$ be the harmonic measure relative to $x$ and $D$, i.e., for every Borel set $A$,
$$
H_D(x,A):=P^x\left(X_{\tau_D}\in A\right).
$$
 For every $f\in \cc{B}_b(\R^d)$, let  $H_Df$ be the function defined on $\R^d$ by
 $$
 H_Df(x)=
\int f(y) H_D(x,dy).
$$
Since, for $x\in D$, the harmonic measure $H_D(x,\cdot)$ is supported by the compact set $\Gamma_D$, it will be convenient to denote again
\begin{equation}\label{hd}
 H_Df(x)=
\int f(y) H_D(x,dy),\quad x\in D,
\end{equation}
for every $f\in \cc{B}_b(\Gamma_D)$.

\par Let $^*\mathcal H^+(\R^d)$ denotes the set of all non negative lower semi-continuous function $f$ on $\R^d$ such that
$$H_Df\leq f\quad\mbox{ for every bounded open set }D.$$
 Since $(\R^d,P^x)$ is a Hunt process, it follows from \cite[Theorem IV.8.1]{hansen} that $(\R^d, ^*\mathcal H^+(\R^d))$ is a balayage space.
Hence, it follows from the general theory of balayage spaces that for every $f\in \cc B_b(\Gamma_D)$
\begin{equation}\label{vd}
H_Df\in C(D)
\end{equation}
and
\begin{equation}\label{vdd}
H_VH_Df=H_Df \quad\textrm{on } V\mbox{ for all open set } V\mbox{ such that }\overline V\subset D.
\end{equation}
Furthermore, a function $f\in \cc B^+(\R^d)$ belongs to $^+\mathcal H^+(\R^d)$ if and only if $\sup_{t>0} P_t^kf=f$.
\par  Let us now introduce the Green function $G^k$ of the Dunkl Laplacian operator which  will play an important role in
our approach. It is defined for every $x,y\in \R^d$ by
$$
G^k(x,y)=\int_0^\infty p_t^k(x,y)dt.
$$
 For every $y\in \R^d$, the function $G^k_y:=G^k(\cdot,y)\in ^*\mathcal H^+(\R^d)$. Indeed, it is not hard to see that
$$P_t^kG^k_y(x)=\int_t^\infty p_s^k(x,y)ds\leq G^k(x,y).$$
Hence, the map $t\mapsto P_t^kG^k_y$ is decreasing on $]0,\infty[$ and so
$$\sup_{t>0}P_t^kG^k_y=\lim_{t\rightarrow 0} P_t^kG^k_y=G^k_y.$$
Whence $G^k_y\in\,^*\mathcal H^+(\R^d)$ which means that
\begin{equation}\label{ghd}
\int G^k(z,y)H_D(x,dz)\leq G^k(x,y).
\end{equation}
Furthermore, it is obvious that $G^k$ is positive and symmetric on $\R^d\times\R^d$. Therefore, it follows from \cite[Theorem VI.1.16]{blum} that for every bounded open set $D$  and   every $x,y\in \R^d$,
\begin{equation}\label{blum}
\int G^k(x,z)H_D(y,dz)=\int G^k(y,z)H_D(x,dz) .
\end{equation}
\section{ Dirichlet problem}

Let $B$ be an open ball of $\R^d$ of center $0$ and radius $r>0$. We first introduce three kinds of harmonicity  on $B$.

A continuous function $h: B\to\R$ is said to be:
\begin{enumerate}
\item $\Delta_k$-harmonic on $B$ if
$
h\in C^2(B)\;\textrm{ and }\;\Delta_kh(x)=0\;\textrm{ for every }\;x\in B.
$
\item $X$-harmonic on $B$ if
$
H_Dh(x)=h(x) \;\textrm{ for every bounded open set }\; D \; \textrm{ such that }\; \overline{D}\subset B\;\textrm{ and every }\;x\in D.
$
\item $\Delta_k$-harmonic on $B$ in the distributional sense if
$$
\langle h, \Delta_k\varphi\rangle_k:=\int_B h(x)\Delta_k\varphi(x)w_k(x)dx=0\quad \textrm{for all}\;\;\varphi\in C^\infty_c(B).
$$
\end{enumerate}

The following two technical lemmas are important in our approach.
\begin{lemma}
Let $f\in C^2_c(\R^d)$. For every $x\in \R^d$,
\begin{equation}\label{inv}
 \int_{\R^d} G^k(x,y)\Delta_kf(y)w_k(y)dy=-f(x).
\end{equation}
In particular, for every bounded open set $D$ and every $x\in D$,
\begin{equation}\label{dynk}
H_Df(x)-f(x)= E^x\left[\int_0^{\tau_D}\Delta_kf(X_s) ds\right].
\end{equation}
\end{lemma}
\begin{proof}
Let $x\in \R^d$. Using Fubini's theorem and formulas (\ref{ti}) and (\ref{chal}),  we have
 \begin{eqnarray*}
\int_{\R^d} G^k(x,y)\Delta_kf(y)w_k(y)\,dy &=& \int_0^\infty \int_{\R^d} p_t^k(x,y) \Delta_kf(y)w_k(y)\,dy\,dt\\
&=& \int_0^\infty \int_{\R^d} \Delta_k\left(p_t^k(x,\cdot)\right)(y) f(y)w_k(y)\,dy\,dt\\
&=& \int_0^\infty \int_{\R^d} \Delta_k\left(p_t^k(\cdot,y)\right)(x) f(y)w_k(y)\,dy\,dt\\
&=& \lim_{t\to\infty}P_t^kf(x)-\lim_{t\to 0}P^k_tf(x)\\
&=& -f(x).
 \end{eqnarray*}
 To get $\lim_{t\to\infty}P_t^kf(x)=0$ we only use (\ref{ptk}) and the fact that $f$ is with compact support.
Formula (\ref{dynk}) follows from (\ref{inv}) and the strong Markov property. In fact, let $D$ be a bounded open set and let $x\in D$. Then
\begin{eqnarray*}
-f(x) &=& \int G^k(x,y)\Delta_kf(y)w_k(y)dy\\
&=& \int_0^\infty\int p_t^k(x,y)\Delta_kf(y)w_k(y)dy dt\\
&=& E^x\left[\int_0^\infty\Delta_kf(X_s) ds\right]\\
&=& E^x\left[\int_0^{\tau_D}\Delta_kf(X_s) ds\right]+E^x\left[\int_{\tau_D}^\infty\Delta_kf(X_s) ds\right]\\
&=& E^x\left[\int_0^{\tau_D}\Delta_kf(X_s) ds\right]+E^x\left[E^{X_{\tau_D}}\left[\int_0^\infty\Delta_kf(X_s) ds\right]\right]\\
&=& E^x\left[\int_0^{\tau_D}\Delta_kf(X_s) ds\right]+E^x\left[-f(X_{\tau_D})\right]\\
&=& E^x\left[\int_0^{\tau_D}\Delta_kf(X_s) ds\right]-H_Df(x).
\end{eqnarray*}
\end{proof}
\begin{lemma}
For every  bounded open set $D$ and
 for every $\varphi,\psi\in C^2_c(\R^d)$,
\begin{equation}\label{td}
    \langle H_D\psi, \Delta_k\varphi\rangle_k = \langle \Delta_k\psi,H_D\varphi\rangle_k.
\end{equation}
\end{lemma}
\begin{proof}
 Applying formula (\ref{inv}) to $\psi$, we get
\begin{equation}\label{psifi}
\langle H_D\psi, \Delta_k\varphi\rangle_k= -\int\int\int G^k(z,y)\Delta_k\psi(y)w_k(y)dyH_D(x,dz)\Delta_k\varphi(x)w_k(x)dx.
\end{equation}
Then (\ref{td}) is obtained by Fubini's theorem using formula (\ref{blum}) and formula (\ref{inv}) applied to $\varphi$.
\end{proof}
Now, we  prove that the three kinds of harmonicity on $B$, introduced in the beginning of this section, are equivalent.
\begin{theorem}\label{fond1}
Let $h\in C(B)$. The following three assertions are equivalent.
\begin{enumerate}
\item $h$ is $\Delta_k$-harmonic on $B$.
\item $h$ is $X$-harmonic on $B$.
\item $h$ is $\Delta_k$-harmonic on $B$ in the distributional sense.
\end{enumerate}
\end{theorem}
\begin{proof}
\begin{enumerate}
\item Assume that $h$ is $\Delta_k$-harmonic on $B$.
Let $D$ be a bounded open set such that $\overline{D}\subset B$ and let $x\in D$. We  claim that
\begin{equation}\label{dyn}
   H_Dh(x)-h(x)=
    E^x\left[\int_0^{\tau_D}\Delta_k h(X_s)ds\right].
\end{equation}
Let $V$ be a bounded open set such that $\overline{D}\subset V\subset\overline{V}\subset B$. By $C^\infty$-Uryshon's lemma, there exists $\theta\in C^\infty_c(B)$ such that $\theta=1$ on $V$. Let $f:=h\theta$ and $\psi:=h-f$. Obviously, $h=f$ on $V$, $\psi=0$ on $V$ and  $f\in C^2_c(B)$.
Then using (\ref{dynk}) we obtain
\begin{equation}\label{dyn1}
H_Dh(x)-h(x)
  =  E^x\left[\int_0^{\tau_D}\Delta_kf(X_s)ds\right]+H_D\psi(x).
\end{equation}
For every $y\in\R^d$, let $N(y,dz)$ be the L\'evy kernel of the Dunkl process $X$ which is given in \cite{gy} by the following formula:
\begin{equation}\label{lvk}
N(y,dz)=
\sum_{\alpha\in R_+, \langle y,\alpha\rangle\neq0}\frac{k(\alpha)}{\langle\alpha,y\rangle^2}\delta_{\sigma_\alpha y}(dz).
\end{equation}
Since $\psi =0$ on $V$, it follows from \cite[Theorem 1]{iw} that
\begin{equation}\label{wat}
H_D\psi(x)=
E^x
 \left[\int_0^{\tau_D}\int \psi(z)N(X_s,dz)ds\right].
\end{equation}
On the other hand, by (\ref{lapd}) and (\ref{lvk})  we easily see that  for every $y\in D$,
\begin{equation}\label{wat1}
\Delta_kf(y)=\Delta_kh(y)-\int \psi(z)N(y,dz).
\end{equation}
Thus formula~(\ref{dyn}) is obtained by combing (\ref{dyn1}), (\ref{wat}) and
(\ref{wat1}) above. Hence, by (\ref{dyn}), $H_Dh(x)=h(x)$ and thereby $h$ is $X$-harmonic on $B$.
\item Assume that $h$ is $X$-harmonic on $B$. Let $\varphi\in C^\infty_c(B)$ and let $D\subset\overline{D}\subset B$ be a $W$-invariant bounded open set which contains the support of $\varphi$. Let $(h_n)_{n\geq 1}\subset~ C^2_c(B)$ be a sequence which converges uniformly to $h$ on $\partial D $. Since $H_D\varphi=0$ on $D$, applying (\ref{td}) we obtain
\begin{equation}\label{disn}
\langle H_Dh_n, \Delta_k\varphi\rangle_k = 0,\quad n\geq 1.
\end{equation}
On the other hand,
\begin{eqnarray*}
\sup_{x\in D}|H_Dh_n(x)-H_Dh(x)|&=&\sup_{x\in D}\left|\int_{\partial D} (h_n(y)-h(y))H_D(x,dy)\right|\\
&\leq&\sup_{y\in \partial D}|h_n(y)-h(y)|\longrightarrow 0\quad\textrm{as}\quad n\longrightarrow\infty.
\end{eqnarray*}
Hence,  by letting $n$ tend to $\infty$ in~(\ref{disn}), $\langle H_Dh, \Delta_k\varphi\rangle_k = 0$ and therefore $\langle h, \Delta_k\varphi\rangle_k = 0$ since $h=H_Dh$ on $D$.
\item Assume that $h$ is $\Delta_k$-harmonic on $B$ in the distributional sense. The hypoellipticity of the Dunkl Laplacian $\Delta_k$ on $W$-invariant open sets \cite{kk1,mt2} yields that $h\in C^\infty(B)$. Thus, using (\ref{ti}), it follows that, for every $\varphi\in C^\infty_c(B)$,
$$
\int_B\Delta_kh(x)\varphi(x)w_k(x)\,dx=0.
$$
Hence $\Delta_kh(x)=0$ for every $x\in B$ which means that $h$ is $\Delta_k$-harmonic on $B$.
\end{enumerate}
\end{proof}
It is worth noting that statements of the above theorem remain true if we replace the ball $B$ by any $W$-invariant open set, namely the whole space $\R^d$.
\begin{theorem}\label{theo1}
For every $f\in C^+(\partial B)$, the problem
\begin{equation}\label{dir}
\left\{\begin{array}{ll}
 \Delta_kh=0\;\textrm{ on}\; B\\
 h=f\;\textrm{on}\; \partial B,
\end{array}
\right.
\end{equation}
admits one and only one solution in $C^+(\overline{B})$ which is given by $H_Bf$.
\end{theorem}
\begin{proof}
Let $f\in C^+(\partial B)$. By (\ref{vd}) and (\ref{vdd}), the function $H_Bf$ is continuous and $X$-harmonic on $B$. We shall show that $H_Bf$ is a continuous extension of $f$ on $\overline{B}$. Let $z\in \partial B$ and consider $V=\R^d\backslash\{0\}$ and $u$ the function defined on $V$ by
$$
u(x)=G^k(x,0)-G^k(z,0).
$$
Since $p_t^k(x,0)= \frac{c_k}{(4t)^{\frac{m}{2}}}
e^{-\frac{|x|^2}{4t}},\; x\in \R^d$, it follows  that
\begin{equation}\label{gk0}
G^k(x,0)=\frac{c_k}{4}\frac{\Gamma(m/2-1)}{|x|^{m-2}}.
\end{equation}
 Then, using (\ref{ghd}) and (\ref{gk0}), it is easy to verify that $u$ is a  barrier of $z$ (with respect to $B$), i.e.,
 \begin{itemize}
 \item [i)] $u$ is hyperharmonic on $V\cap B$,
 \item [ii)] $u$ is positive on $V\cap B$,
 \item [iii)]$\lim_{x\in V\cap B,x\rightarrow z}u(x)=0$.
 \end{itemize}
 Hence, by \cite[Propositions VII.3.1 and VII.3.3]{hansen}, $H_B(z,\cdot)=\delta_z$ and $\lim_{x\in B,x\rightarrow z}H_B f(x)=f(z)$. Since $z$ is arbitrary in $\partial B$, we conclude that $H_Bf$ is a continuous extension of $f$ on $\overline{B}$.
 So, it remains to prove the uniqueness. Let $h$ be an other continuous extension of $f$ on $\overline{B}$ solution to the problem (\ref{dir}). Let $x\in B$ and let $(D_n)_{n\geq 1}$ be a sequence of nonempty bounded open sets such that $x\in  D_n\subset\overline{D_n}\subset D_{n+1}$ and $B=\cup_nD_n$. Then  $(\tau_{D_n})_n$ converges to $\tau_B$ almost surely. Hence, the continuity of $h$ on $\overline{B}$ together with the quasi-left-continuity of the Dunkl process yield that $H_Bh(x)=\lim_nH_{D_n}h(x)$ and consequently $H_Bh(x)=h(x)$ since $H_{D_n}h(x)=h(x)$ for every $n\geq 1$. Thus $h(x)-H_Bf(x)=H_B(h-f)(x)=0$ since $h=f$ on $\partial B$. So, $h=H_Bf$ on $B$ and the uniqueness is proved.
\end{proof}
\section{Green operators}
The Green operator $G^k$ on the whole space $\R^d$ is defined, for every $f\in \cc{B}^+(\R^d)$, by
$$
G^kf(x):=\int_{\R^d}G^k(x,y)f(y)w_k(y)\,dy,\;x\in\R^d.
$$
By Fatou's lemma,  for each $y\in \R^d$, $G^k(\cdot,y)$ is lower semi-continuous on $\R^d$ and so $G^kf$ is lower semi-continuous on~$\R^d$.

In the sequel,   $B_r$ denotes the ball of $\R^d$ of center $0$ and radius $r>0$ and  $A_{t,s}$ denotes the annulus of $\R^d$ of center $0$ and radius $0\leq t<s<\infty$.
\begin{lemma}
\begin{enumerate}
\item
For every $0<r<\infty$,
\begin{equation}\label{g1b}
G^k1_{B_r}(x)=
\left\{\begin{array}{cr}
\frac{1}{m-2}\left(\frac{|x|^2}{m}+\frac{r^2-|x|^2}{2}\right)&\textrm{if}\;\; |x|\leq r\\
 \frac{1}{m(m-2)}r^m|x|^{2-m}&\textrm{if}\;\; |x|\geq r.
\end{array}
\right.
\end{equation}
\item
For every $0\leq t<s<\infty$,
\begin{equation}\label{g1a}
0\leq \sup_{x\in A_{t,s}}G^k1_{A_{t,s}}(x)\leq \frac{2}{m-2}s(s-t).
\end{equation}
\end{enumerate}
\end{lemma}
\begin{proof}Formula (\ref{g1b}) follows immediately from (\ref{locb1}) seeing that
$$
G^k1_{B_r}(x)=\int_0^\infty \int_{B_r} p_t^k(x,y)w_k(y)\,dy\, dt.
$$
Let $0\leq t<s<\infty$. It is clear that $0\leq G^k1_{A_{t,s}}$ and that
$$
G^k1_{A_{t,s}}= G^k1_{B_{s}}-G^k1_{B_{t}}.
$$
Then, using (\ref{g1b}), it follows that for every $x\in A_{t,s}$,
\begin{eqnarray*}
G^k1_{A_{t,s}}(x) & = & \frac{1}{m-2}\left[\frac{|x|^2}{m}+\frac{s^2-|x|^2}{2}\right]- \frac{1}{m(m-2)}t^m|x|^{2-m}\\
&=& \frac{1}{m-2}\left[\frac{|x|^2}{m}\left(1-\left(\frac{t}{|x|}\right)^m\right)+ \frac{s^2-|x|^2}{2}\right]\\
&\leq& \frac{1}{m-2}\left[\frac{s^2}{m}\left(1-\left(\frac{t}{s}\right)^m\right)+ \frac{s^2-t^2}{2}\right]\\
&\leq& \frac{1}{m-2}\left[s^2\left(1-\frac ts\right)+ \frac{s^2-t^2}{2}\right]\\
&\leq& \frac{2}{m-2}s(s-t).
\end{eqnarray*}
\end{proof}

An immediate consequence of the above lemma is that for each $x\in\R^d$ the function $G^k(\cdot,x) w_k$ is locally Lebesgue-integrable  on $\R^d$. Thus, by Fubini's theorem, for every  $f\in\cc{B}_b(\R^d)$ with compact   support, we have
\begin{eqnarray*}
G^k f(x)= \int_{\R^d} G^k(x,y) f(y)w_k(y)dy&=&\int_0^\infty \int_{\R^d} p_t^k(x,y)f(y) w_k(y)dydt\\
&=& \int_0^\infty E^x\left[f(X_t)\right]dt=E^x\left[\int_0^\infty  f(X_t) dt\right].
\end{eqnarray*}
\begin{proposition}
Let $f\in \cc{B}_b(\R^d)$ with compact support. Then $G^kf\in C_0(\R^d)$ and
\begin{equation}\label{p1}
\Delta_k G^kf=-f \;\textrm{ in }\;\R^d
\end{equation}
in the distributional sense, i.e., for every $\psi\in \mathcal{C}_c^\infty(\R^d)$,
$$
 \int_{\R^d} G^kf(x)\Delta_k\psi(x)w_k(x)\, dx=  -\int_{\R^d} f(x)\psi(x)w_k(x)\,dx.
$$
Moreover, if $f$ is radially symmetric then $G^kf$ is also.
\end{proposition}
\begin{proof}

Let $r>0$ such that the support of $f$ is contained in $B_r$. Let us first assume that $f\geq0$ and let $g=\|f\|\,1_{B_r}-f$.
Then applying the Green operator $G^k$, we obtain
\begin{equation}\label{gkfg}
G^kf+G^kg=\|f\|\,G^k1_{B_r}.
\end{equation}
Since  $G^kf$ and $G^kg$ are lower semi-continuous on $\R^d$ and $G^k1_{B_r}\in C_0(\R^d)$ (see~(\ref{g1b})), we immediately deduce from (\ref{gkfg}) that $G^kf\in C_0(\R^d)$. For $f$  of arbitrary sign, we write $f=f^+-f^-$  where $f^+=\max(f,0)$ and $f^-=\max(-f,0)$. Then the same reasoning shows that $G^kf^+$ and $G^kf^-$ are in $C_0(\R^d)$. Hence $G^kf=G^kf^+-G^kf^-$ is in $C_0(\R^d)$ as desired.
Let $\psi\in \mathcal{C}_c^\infty(\R^d)$. Then, by (\ref{inv}), for every $y\in \R^d$ we have
$$
\int_{\R^d} G^k(x,y)\Delta_k\psi(x)w_k(x) = -\psi(y).
$$
Hence,
\begin{eqnarray*}
 \int_{\R^d} G^kf(x)\Delta_k\psi(x)w_k(x)\, dx
&=& \int_{\R^d} \left(\int_{\R^d} G^k(x,y)f(y)w_k(y)\,dy\right)\Delta_k\psi(x)w_k(x)\, dx\\
&=& \int_{\R^d} \left(\int_{\R^d} G^k(x,y)\Delta_k\psi(x)w_k(x)\, dx\right)f(y)w_k(y)\,dy\\
&=& -\int_{\R^d} f(y)\psi(y)w_k(y)\,dy.
\end{eqnarray*}
Formula (\ref{g1b}) justify the transformation of the above integrals  by Fubini's theorem. Now, assume that $f$ is radially symmetric. Let $(f_n)_n$ be an increasing sequence of functions of the form
$$
f_n=\sum_{i=1}^n\alpha_i1_{B_{r_i}},
$$
which converges pointwise to $f$ on $\R^d$. Clearly, by formula (\ref{g1b}), $G^kf_n$ is radially symmetric. On the other hand, using the dominated convergence theorem, for every $x\in \R^d$, $\lim_{n\to\infty}G^kf_n(x)=G^kf(x)$. Thus $G^kf$ is  radially symmetric.
\end{proof}

For every open set $D$, we define the Green operator $G^k_D$ on $\cc{B}_b(D)$ by
$$
G^k_Df(x):=E^x\left[\int_0^{\tau_D}f(X_s)\,ds\right],\quad x\in D.
$$
For every $f\in\cc{B}_b(D)$, we denote by $\widetilde{f}$  the extension of $f$ on $\R^d$ such that $\widetilde{f}=0$ on $\R^d\setminus D$. Since Dunkl process satisfies the strong Markov property, for every $x\in D$ we have
\begin{eqnarray*}
G^k\widetilde{f}(x) &=& E^x\left[\int_0^\infty\widetilde{f}(X_s)\,ds\right]\\
 &=& E^x\left[\int_0^{\tau_D}\widetilde{f}(X_s)\,ds\right]+ E^x\left[\int_{\tau_D}^\infty\widetilde{f}(X_s)\,ds\right]\\
  &=& E^x\left[\int_0^{\tau_D}f(X_s)\,ds\right] + E^x\left[E^{X_{\tau_D}}\left[\int_0^\infty\widetilde{f}(X_s) ds\right]\right]\\
&=& E^x\left[\int_0^{\tau_D}f(X_s)\,ds\right] + H_DG^k\widetilde{f}(x).
\end{eqnarray*}
Therefore,
\begin{equation}\label{gkb}
G^k_Df=G^k\widetilde{f}-H_DG^k\widetilde{f}\; \textrm{ on }\; D.
\end{equation}
Let $B$ be an open ball of $\R^d$ of center $0$ and radius $r>0$. Then it follows from (\ref{gkb}) that, for every $f\in \cc{B}_b(B)$, $G^k_Bf$ can be represented by
$$
G^k_Bf(x)=\int_BG^k_B(x,y)f(y) w_k(y)\, dy,
$$
where, for every $x,y\in B$,
\begin{equation}\label{gkbxy}
G^k_B(x,y):=G^k(x,y)-\int_{\partial B} G^k(y,z)H_B(x,dz).
\end{equation}
Since, by (\ref{ptk}), for every  $y,z\in \R^d$, we have
\begin{equation}\label{gin}
G^k(y,z)\leq \frac{c_k\Gamma(m/2-1)}{4\left(|y|-|z|\right)^{m-2}},
\end{equation}
it is immediate to see that, for every $x,y\in B$,
$$
\int_{\partial B} G^k(y,z)H_B(x,dz)\leq \frac{c_k\Gamma(m/2-1)}{4\left(|y|-r\right)^{m-2}}<\infty.
$$
Therefore, for every $x,y\in B$, $G^k_B(x,y)$ introduced in (\ref{gkbxy}) is well defined. In the following corollary, we collect some properties of the Green operator $G^k_B$.
\begin{corollary}
Let  $f\in \cc{B}_b(B)$.
Then $G^k_Bf\in C_0(B)$ and
\begin{equation}\label{invb}
\Delta_kG^k_Bf=-f \;\textrm{ in }\; B
\end{equation}
in the distributional sense.
Moreover, if $f$ is radially symmetric then $G^k_Bf$ is also.
\end{corollary}
\begin{proof}
Clearly, $G^k_Bf$ is continuous on $B$ since $G^k\widetilde{f}$ and $H_BG^k\widetilde{f}$ are. For every $z\in \partial B$,
$$
\lim_{x\to z}G^k_Bf(x)=0,
$$
since  $\lim_{x\to z}H_BG^k\widetilde{f}(x)= G^k\widetilde{f}(z)$. Thus $G^k_Bf\in C_0(B)$. Formula (\ref{invb}) follows immediately from (\ref{p1}) and (\ref{gkb}). If $f$ is radially symmetric on $B$ then $G^k\widetilde{f}$ is also. Therefore, by (\ref{gkb}), $G^k_Bf(x)=G^k\widetilde{f}(x)-G^k\widetilde{f}(z)$ for some $z\in \partial B$. Hence, $G^k_Bf$ is radially symmetric on $B$ as desired.
\end{proof}
\begin{proposition}\label{compact}
For every $M>0$, the family $\{ G^k_Bf,\;\|f\|\leq M \}$ is relatively compact in $C_0(B)$ endowed with
 the uniform  norm.
\end{proposition}
\begin{proof}
In virtue of Arzel\`a-Ascoli theorem, we need to show that $\{ G^k_Bf,\;\|f\|\leq M \}$ is uniformly bounded and equicontinuous on $B$.
Denote by $r$  the radius of the ball $B$. Let $f\in \cc{B}_b(B)$ such that $\|f\|\leq M$. Obviously, $\|G^k_Bf\|\leq M \|G^k_B1\|\leq M\|G^k1_B\|$. Thus, using (\ref{g1b}), we obtain
$$
\|G^k_Bf\|\leq\frac{r^2M}{2(m-2)}.
$$
This means that the family $\{ G^k_Bf,\;\|f\|\leq M \}$ is uniformly bounded. Next, we claim that the family
$\{G_B^k(x,\cdot),\; x\in B\}$ is uniformly integrable with respect to the measure $w_k(y)\,dy$. Let $x\in B$ and $\epsilon>0$ small enough. Let $A_{t,s}$ be the annulus of $\R^d$ of center $0$ and radius $t=\max(0,|x|-\epsilon)$ and $s=|x|+\epsilon$. Then, for every Borel subset $D$ of $B$, we have
\begin{eqnarray*}
\int_D G^k_B(x,y)w_k(y)dy&\leq & \int_D G^k(x,y)w_k(y)dy\\
 &=& \int_{D\cap A_{t,s}} G^k(x,y)w_k(y)\,dy+ \int_{D\setminus A_{t,s}} G^k(x,y)w_k(y)dy\\
 &\leq& G^k1_{A_{t,s}}(x)+\left(\sup_{y\in D\setminus A_{t,s}}G^k(x,y)\right)\int_{D}w_k(y)dy.
\end{eqnarray*}
Hence, it follows from (\ref{gin}) and (\ref{g1a}) that
$$
\int_D G^k_B(x,y)w_k(y)\,dy\leq \frac{4r}{m-2}\epsilon+ \frac{c_k\Gamma(m/2-1)}{4\epsilon^{m-2}}\,\int_D w_k(y)\,dy.
$$
Put $\eta=\epsilon^{m-1}$. Then for every Borel subset $D$ of $B$ such that $\int_D w_k(y)\,dy<\eta$, we have
$$
\int_D G^k_B(x,y)w_k(y)\,dy\leq \left(\frac{4r}{m-2}+ \frac{c_k\Gamma(m/2-1)}{4}\right)\epsilon.
$$
Thus, the uniform integrability of the family $\{G_B^k(x,\cdot),\; x\in B\}$ is shown. Therefore, in virtue of Vitali's convergence theorem, for $z\in B$,
$$
\lim_{x\to z}\int_B\left|G^k_B(x,y)-G^k_B(z,y)\right|w_k(y)\,dy = 0.
$$
Hence, the family $\{G_B^kf,\; \|f\|\leq M\}$ is equicontinuous on $B$ since
$$
\lim_{x\to z}\sup_{\|f\|\leq M}\left|G^k_Bf(x)-G^k_Bf(z)\right|\leq M\lim_{x\to z}\int_B\left|G^k_B(x,y)-G^k_B(z,y)\right|w_k(y)\,dy = 0.
$$
\end{proof}
\section{Semilinear Dirichlet problem}
Let $B$ be an open ball of $\R^d$ of center $0$. Let $\varphi: [0,\infty[\to[0,\infty[$ be a locally Lipschitz nondecreasing function such that
$\varphi(0)=0$.
By a solution of
\begin{equation}\label{eqb}
\Delta_k u=\varphi(u)\;\textrm{ in }\; B,
\end{equation}
we shall mean every function $u\in C(B)$ such that
$$
\int_{B} u(x)\,\Delta_k \psi(x)\, w_k(x)\,dx=\int_B \varphi(u(x))\,\psi(x)\,w_k(x)\,dx
$$
holds for every  $\psi\in \mathcal{C}_c^\infty(B)$.
 We recall from Theorem~\ref{theo1} that if $\varphi\equiv0$ then $H_Bf$ is the unique solution of~(\ref{eqb}) satisfying $u=f$ on $\partial B$. In all the following, we assume that $\varphi$ does not vanish identically.
\begin{lemma}\label{carac}
Let $u\in C(\overline{B})$. Then $u$ is a solution of Eq. (\ref{eqb}), if and only if, $u+G^k_B(\varphi(u))=H_Bu$ on $B$.
\end{lemma}
\begin{proof}
put $h:= u+G^k_B(\varphi(u))$. Clearly, $h\in C(\overline{B})$ and $h=u$ on $\partial B$. For every $\psi\in C_c^\infty(B)$, using Fubini's theorem and formula (\ref{invb}), we obtain
\begin{eqnarray*}
\int_B h(x)\Delta_k\psi(x)w_k(x)\,dx &=& \int_B u(x) \Delta_k \psi(x) w_k(x)\, dx + \int_BG^k_B(\varphi(u))(x)\Delta_k\psi(x)w_k(x)\,dx\\
&=& \int_B u(x) \Delta_k \psi(x) w_k(x)\, dx - \int_B\varphi(u(x))\psi(x) w_k(x)\,dx.
\end{eqnarray*}
So,  $\Delta_k u=\varphi(u)$ in $B$ if and only if $\Delta_k h=0$ in $B$. In this case, since $h=u$ on $\partial B$, the uniqueness of solution to problem (\ref{dir}) yields $h=H_Bu$ on $B$. This completes the proof.
\end{proof}
\begin{lemma}\label{comp}
Let $u,v\in C(\overline{B})$ be two solutions of Eq. (\ref{eqb}). If $u\geq v$ on $\partial B$, then $u\geq v$ on $B$.
\end{lemma}
\begin{proof}
Define $w :=u-v$ and $\rho:= \varphi(u)-\varphi(v)$.   By the above lemma we have
\begin{equation}\label{rh}
w+G^k_B\rho=H_Bw \;\textrm{ on }\; B.
\end{equation}
Suppose that the open set
$
D:=\{x\in B;\; w(x)<0\}
$
is not empty.
Let $x\in D$. Clearly, $B$ contains the support of the measure $H_D(x,\cdot)$. Then applying $H_D(x,\cdot)$ to (\ref{rh}) we get
$$
H_Dw(x)+H_D\left(G^k_B\rho\right)(x)=H_DH_Bw(x)=H_Bw(x).
$$
Consequently,
\begin{equation}\label{rho}
H_Dw(x)= H_Bw(x)- H_D\left(G^k_B\rho\right)(x)= w(x)+\left(G^k_B\rho(x)-H_DG^k_B\rho(x)\right).
\end{equation}
On the other hand, using the strong Markov property,
\begin{equation}\label{mark}
G^k_B\rho(x)-G^k_D\rho(x)=E^x\left[\int_{\tau_D}^{\tau_B}\rho(X_s)\,ds\right]=E^x\left[E^{X_{\tau_D}}\left[\int_0^{\tau_B}\rho(X_s)\,ds\right]\right]
=H_DG^k_B\rho(x).
\end{equation}
Thus, it follows from (\ref{rho}) and (\ref{mark}) that $w(x)+G^k_D\rho(x)=H_Dw(x)$. But this is absurd since $w(x)+G^k_D\rho(x)<0$ and $H_Dw(x)\geq0$. Therefore, $D$ is empty and consequently $u\geq v$ on $B$.
\end{proof}
\begin{theorem}
For every $f\in C^+(\partial B)$, the semilinear Dirichlet problem
\begin{equation}\label{ndp}
\left\{
\begin{array}{rlll}\Delta_k u&=&\varphi(u)& \mbox{in}\ B,\\
u&=&f&\mbox{on}\,\partial B
\end{array}\right.
\end{equation}
admits one and only one solution $u\in C^+(\overline{B})$.
\end{theorem}
\begin{proof}
It follows from Lemma \ref{comp} that problem (\ref{ndp}) admits at most one solution. To prove the existence, in virtue of Lemma \ref{carac}, it will be sufficient to establish the existence of $u\in C^+(\overline{B})$ such that
\begin{equation}\label{ugkb}
u+G^k_B(\varphi(u))=H_Bf\; \textrm{ on }\; B.
\end{equation}
Since $G_B^k1\leq G^k1_B$, we immediately deduce using (\ref{g1b}) that $\sup_{x\in B}G_B^k1(x)<\infty$. Let $f\in C^+(\partial B)$, $a=\|f\|$ and $M= a+\varphi(a)\sup_{x\in B}G_B^k1(x)$.
Let $\phi$ be the function defined on $\R$ by
$$
\phi(t)=
\left\{
  \begin{array}{ll}
    0, & \hbox{if}\; t\leq 0 \\
    \varphi(t), & \hbox{if}\;0\leq t\leq a \\
    \varphi(a), & \hbox{if}\; t\geq a.
  \end{array}
\right.
$$
Let  $\Lambda:=\{u\in C(\overline{B});\;\|u\|\leq M\}$ and consider the operator $T: \Lambda\to C(\overline{B})$ defined  by
$$
Tu(x)=H_Bf(x)-G^k_B(\phi(u))(x),\; x\in \overline{B}.
$$
Since $0\leq\phi(t)\leq \varphi(a)$ for every $t\in\R$, we obtain
$$
|Tu(x)|\leq M
$$
for every $u\in \Lambda$ and every $x\in B$. This implies that $T(\Lambda)\subset\Lambda$. Now, let $(u_n)_n$ be a sequence in $\Lambda$ converging uniformly to $u\in \Lambda$. Let $\varepsilon>0$. Seeing that $\phi$ is uniformly continuous on the interval $[-M,M]$, we immediately deduce that there exists $n_0\in \N$ such that, for every $n\geq n_0$,
$$
\|\phi(u_n)-\phi(u)\|\leq\varepsilon.
$$
Then, for every $n\geq n_0$ and every $x\in B$,
$$
|Tu_n(x)-Tu(x)|\leq G^k_B\left(|\phi(u_n)-\phi(u)|\right)(x)\leq \varepsilon\,\sup_{x\in B}G_B^k1(x).
$$
This show that $(Tu_n)_n$ converges uniformly to $Tu$ and therefore $T$ is a continuous operator. Since $\Lambda$ is a closed bounded convex subset of $C(\overline{B})$ and, in virtue of Proposition \ref{compact}, $T(\Lambda)$ is relatively compact, the Schauder’s fixed point theorem ensures the existence of a function $u\in \Lambda$
such that
$$
u+G^k_B(\phi(u))=H_Bf\; \textrm{ on }\; B.
$$
Clearly $u\in C(\overline{B})$ and $u(x)\leq H_bf(x)\leq a$ for every $x\in B$. So, to obtain (\ref{ugkb}), we need to show that $\phi(u)=\varphi(u)$ on $B$, or equivalently, $u\geq 0$ on $B$. Assume that the open set $D:=\{x\in B,\;u(x)<0\}$ is not empty.  Let $x\in D$.  Then,
$$
H_Du(x)= H_D\left(H_Bu-G^k_B(\phi(u))\right)(x)
=H_Bu(x)-H_DG^k_B(\phi(u))(x).
$$
The same reasoning as in (\ref{mark}), based on the strong Markov property, shows that
$$
H_DG^k_B(\phi(u))(x)=G^k_B(\phi(u))(x)-G^k_D(\phi(u))(x).
$$
Thus, seeing that $\phi(u)=0$ on $D$,
\begin{eqnarray*}
H_Du(x)&=& H_Bu(x)-G^k_B(\phi(u))(x)+G^k_D(\phi(u))(x)\\
&=& u(x)+G^k_D(\phi(u))(x)\\
&=& u(x)<0.
\end{eqnarray*}
But, $H_Du(x)\geq 0$ since $u\geq 0$ on $B\setminus D$ which contain the support of $H_D(x,\cdot)$.
Hence $D$ must be empty and consequently $u\geq 0$ on $B$.
\end{proof}

\section{Keller-Osserman condition}
Let $\varphi:\R_+\rightarrow
\R_+$ be a locally Lipschitz  nondecreasing function on $\R_+$ such that
$\varphi(0)=0$.
Our purpose now consists in studying the existence of positive solution (in the distributional sense)
of the following equation
\begin{equation}\label{eeq}
\Delta_kv=\varphi(v)
\end{equation}
in the whole space $\R^d$. Such a solution will be called entire positive solution.

Let  $v\in C^2(\R^d)$ be radially symmetric on $\R^d$, i.e., $v(x)=u(|x|)$ for some function $u\in C^2(\R)$. Then, for every $\alpha\in R$,
$$
\langle\nabla v(x),\alpha\rangle = \frac{u'(|x|)}{|x|}\,\langle x,\alpha\rangle \quad \textrm{ and } \quad v(\sigma_\alpha x)=v(x).
$$
Thus, it follows from (\ref{lapd}) that
$$
\Delta_kv(x)=u''(|x|)+\frac{m-1}{|x|}u'(|x|).
$$
Hence,
$$
\Delta_kv(x)=\varphi(v(x)) \;\textrm{ if and only if }\; u''(|x|)+\frac{m-1}{|x|}u'(|x|)=\varphi(u(|x|)).
$$

It is well known from the general theory of ordinary differential equation that, for every $a\geq0$,
the equation
\begin{equation}\label{ord}
u''+\frac{m-1}{r}u'=\varphi(u)
\end{equation}
admits one and only one positive solution $u_a$  defined on a maximal interval $[0,R_a[$ such that $u_a'(0)=0$ and $u_a(0)=a$. Moreover, $u_0(r)=0$ for all $r\geq0$.
Writing (\ref{ord}) in the form $(r^{m-1}u_a'(r))'=r^{m-1}\varphi(u_a(r))$ and then integrating twice from $0$ to $r$, we obtain
\begin{equation}\label{inteq}
u_a(r)=a+\int_0^rt^{1-m}\int_0^ts^{m-1}\varphi\left(u_a(s)\right)\,ds\,dt\quad\textrm{ for all }\; r\in[0,R_a[.
\end{equation}
Moreover, for every $a>0$,
$$\lim_{r\to R_a}u_a(r)=\infty.$$
 Indeed, if $R_a<\infty$ then, under the maximality condition, we must have $u_a(r)\to\infty$ as $r\to R_a$. If $R_a=\infty$ then it follows from the above integral equation that, for every $r\in [0,\infty[$, $u_a(r)\geq a+\frac{\varphi(a)}{2m}r^2$  and therefore $u_a(r)\to\infty$ as $r\to \infty$.
 \begin{lemma}\label{uab}
\begin{enumerate}
\item For every $0\leq a\leq b$, we have $u_a\leq u_b$ and $R_a\geq R_b$.
\item For every $0\leq a$, we have  $\lim_{b\to a}R_b=R_a$.
\end{enumerate}
\end{lemma}
\begin{proof}
\begin{enumerate}
\item
Let $0<a<b$ and assume that $u_a(r)=u_b(r)$ for some $r>0$. Since $u_a(0)=a<b=u_b(0)$, without loss of generality we can assume that $u_a(t)<u_b(t)$ for all $0\leq t<r$. Then, using the fact that $\varphi$ is increasing, we obtain
$$
a+\int_0^rt^{1-m}\int_0^ts^{m-1}\varphi\left(u_a(s)\right)\,ds\,dt <b+ \int_0^rt^{1-m}\int_0^ts^{m-1}\varphi\left(u_b(s)\right)\,ds\,dt.
$$
Thus, $u_a(r)<u_b(r)$ contradicting $u_b(r)=u_a(r)$. Since $u_a(R_a)=\infty$ and $u_a\leq u_b$, then we must have $R_a\geq R_b$.
\item Let $0\leq a$ and denote $u=\lim_{b\to a}u_b$. Then it follows from (\ref{inteq}) that for $r$ small enough,
 $$
 u(r)=a+\int_0^rt^{1-m}\int_0^ts^{m-1}\varphi\left(u(s)\right)\,ds\,dt.
 $$
 This means that $u=u_a$ and hence $\lim_{b\to a}R_b=R_a$.
\end{enumerate}
\end{proof}

\begin{lemma}
For every $a>0$, we have
\begin{equation}
\label{ko}
\int_a^\infty\frac{dt}{\sqrt{\int_a^t\varphi(s)ds}}\leq \sqrt 2 R_a \leq \sqrt{m}\int_a^\infty\frac{dt}{\sqrt{\int_a^t\varphi(s)ds}}.
\end{equation}
\end{lemma}
\begin{proof}
Writing (\ref{ord}) in the form
$(r^{m-1}u_a'(r))'=r^{m-1}\varphi(u_a(r))$ and then integrating from $0$ to $r$, we get
$$
u_a'(r)=r^{1-m}\int_0^rs^{m-1}\varphi(u_a(s))\,ds.
$$
This shows that $u_a'\geq 0$ and thereby $u_a$ is  nondecreasing  on $[0,R_a]$. We also deduce, using the fact that $u_a$ and $\varphi$ are nondecreasing, that
$$
u_a'(r)\leq r^{1-m}\varphi(u_a(r))\int_0^rs^{m-1}ds = \frac rm \varphi(u_a(r)).
$$
Combining this inequality with $u_a'\geq 0$, it follows from (\ref{ord}) that
$$
\frac{\varphi(u_a(r))}{m}\leq u_a''(r)\leq \varphi(u_a(r)).
$$
Next, multiplying the last inequalities by $u_a'(r)$ and then integrate from $0$ to $r$, we obtain
$$
\frac{1}{m}\int_{a}^{u_a(r)}\varphi(s)\,ds\leq \frac{(u_a'(r))^2}{2}\leq \int_{a}^{u_a(r)}\varphi(s)\,ds
$$
which is equivalent to
$$
\frac{u_a'(r)}{\sqrt{\int_{a}^{u_a(r)}\varphi(s)\,ds}}\leq \sqrt 2 \leq \sqrt{m}\frac{u_a'(r)}{\sqrt{\int_{a}^{u_a(r)}\varphi(s)\,ds}}.
$$
Finally, integrating again from $0$ to $r$ and then making $r$ tends to $R_a$, we obtain (\ref{ko}).
\end{proof}

\begin{lemma}
Assume that there exists some constant $a>0$ such that
\begin{equation}\label{negcond}
\int_a^\infty\frac{dt}{\sqrt{\int_0^t\varphi(s)ds}}<\infty.
\end{equation}
 Then for every $r>0$ there exists $b>0$ such that $r=R_b$.
\end{lemma}
\begin{proof}
We first note that, for every $\rho>0$, $R_\rho<\infty$ and thus, by the second assertion in Lemma \ref{uab}, the function $\rho\to R_\rho$ is continuous on the interval $]0,\infty[$. So, to complete the proof, it  suffices to show that
\begin{equation}\label{0infty}
\lim_{\rho\to\infty}R_\rho=0 \;\textrm{ and } \; \lim_{\rho\to 0}R_\rho=\infty.
\end{equation}
Since $\varphi$ is locally Lipschitz and $\varphi(0)=0$, there exist $\eta>0$ and $c>0$ such that
$$
\int_0^t\varphi(s)\,ds\leq  c\, t^2\; \textrm{ for all }\; 0\leq t\leq \eta.
$$
This leads to
$$
\lim_{\rho\to 0^+}\int_\rho^\infty\frac{dt}{\sqrt{\int_0^t\varphi(s)ds}}=\int_0^\infty\frac{dt}{\sqrt{\int_0^t\varphi(s)ds}}=\infty
$$
and hence, by (\ref{ko}), $\lim_{\rho\to 0}R_\rho=\infty$. On the other hand,
we immediately deduce from the condition (\ref{negcond}) that $\varphi(t)/t$ increases to $\infty$ when $t$ tends to $\infty$. Thus, by writing
$$
\int_\rho^\infty\frac{dt}{\sqrt{\int_\rho^t\varphi(s)ds}}= \int_1^\infty\frac{dt}{\sqrt{\int_1^t\frac{\varphi(\rho s)}{\rho}ds}},
$$
we obtain using the dominated convergence theorem that
$$\lim_{\rho\to\infty}\int_\rho^\infty\frac{dt}{\sqrt{\int_\rho^t\varphi(s)ds}}=0$$
 and hence, by (\ref{ko}), $\lim_{\rho\to\infty}R_\rho=0$.
\end{proof}
\begin{theorem}
Let $r>0$. Then, the problem
\begin{equation}\label{dirinfty}
\left\{
\begin{array}{rlll}\Delta_k u&=&\varphi(u)& \mbox{in}\ B_r,\\
u&=&\infty&\mbox{on}\;\partial B_r.
\end{array}\right.
\end{equation}
admits a positive solution if and only if (\ref{negcond}) holds for some constant $a>0$.
\end{theorem}
\begin{proof}
Assume that (\ref{negcond}) holds for some constant $a>0$. Then, by the previous lemma, there exists $b>0$ such that $r=R_b$. We denote $v_b$ the function defined on the open ball $B_r$ by $v_b(x)=u_b(|x|)$. Since $v_b\in C^2(B_r)$, it follows from (\ref{ti}) that $\Delta_k v_b=\varphi(v_b)$ in the distributional sense. Hence, $v_b$ is a solution of the problem (\ref{dirinfty}). Conversely, let $u$ be a positive solution of (\ref{dirinfty}). Proceeding by contradiction, assume that (\ref{negcond}) does not holds for all $a>0$, that is,
$$
\int_a^\infty\frac{dt}{\sqrt{\int_0^t\varphi(s)ds}}=\infty \;\textrm{ for all }\; a>0.
$$
Then it follows from (\ref{ko}) that $R_a=\infty$ for all $a>0$. Thus, for every $a>0$, the function $v_a$ defined on $\R^d$ by $v_a(x)=u_a(|x|)$ is an entire positive solution of Eq. (\ref{eeq}). By Lemma \ref{comp}, $v_a\leq u$ on $B_r$ for all $a>0$. In particular, $a=v_a(o)\leq u(0)$ for all $a>0$ and hence, by letting $a$ tend to $\infty$, $u(0)=\infty$ which is impossible. Hence (\ref{negcond}) holds true for some constant $a>0$.
\end{proof}

\begin{theorem}
Eq. (\ref{eeq}) admits an entire positive solution, if and only if, $\varphi$ satisfies the Keller-Osserman condition, i.e., there exists $a>0$ such that
$$
\int_a^\infty\frac{dt}{\sqrt{\int_0^t\varphi(s)ds}}=\infty.
$$
\end{theorem}
\begin{proof}
  Assume that $\varphi$ satisfies the Keller-Osserman condition.  Then, by (\ref{ko}), $R_a=\infty$ and therefore the function $x\mapsto u _a(|x|)$ is an entire positive solution of (\ref{eeq}) as desired. Conversely, Assume that $\varphi$ does not satisfies the Keller-Osserman condition, i.e., for every $a>0$,
$$
\int_a^\infty\frac{dt}{\sqrt{\int_0^t\varphi(s)ds}}<\infty.
$$
As shown in (\ref{0infty}) , $\lim_{\rho\to 0}R_\rho=\infty$. Let $(a_n)_n$ be a strictly monotone
decreasing sequence with limit $0$ chosen so that $(R_{a_n})_n$ be a strictly monotone
increasing sequence. For every integer $n$, we denote $R_n:=R_{a_n}$ and $u_n:=u_{a_n}$.
Then it follows from Lemma \ref{uab} that
$$
u_{n+1}\leq u_n\;\textrm{ on }\; [0,R_n[ \;\textrm{ and } \; \inf_nu_n(r)=0\,\textrm{ for all }\,r\geq 0.
$$
 Now, assume that Eq. (\ref{eeq}) admits an entire nonegative solution $v$. Then a careful application of Lemma \ref{comp} shows that, for every $n$,
\begin{equation}\label{vunb}
v(x)\leq u_n(|x|)\;\textrm{ for }\; |x|<R_n.
\end{equation}
In fact,
 let $n$ be fixed and put $\theta_n:=\sup_{|x|\leq R_n}v(x)$. Since $\lim_{|x|\to R_n}u_n(|x|)=\infty$, there exists $\eta_0>0$ such that $u_n(|x|)\geq \theta_n$ for all $|x|\geq R_n-\eta_0$.
In particular, for every $0<\eta<\eta_0$, we have $u_n(R_n-\eta)\geq \theta_n$  and thereby
$$
u_n(|x|)\geq v(x)\;\textrm{ for }\; |x|=R_n-\eta.
$$
 Then applying Lemma \ref{comp} to $v$ and $u_n(|\cdot|)$ on $B(0, R_n-\eta)$, we deduce that $v(x)\leq u_n(|x|)$ for  $|x|< R_n-\eta$ which yields (\ref{vunb}) since $\eta$ is arbitrarily small. \\
 Let $x\in \R^d$ and let $n_0$ be the smallest  integer such that $|x|<R_{n_0}$. Obviously, $|x|<R_{n_0}<R_n$ for every $n> n_0$. Then it follows from (\ref{vunb}) that, for every $n\geq n_0$,
 $$
 v(x)\leq u_n(|x|).
 $$
 This yields, by letting $n$ tend to $\infty$,  that  $v(x)\leq \inf_nu_n(|x|)=0$. Therefore $v=0$ on $\R^d$ which complete the proof of the theorem.
\end{proof}

\end{document}